\documentclass[12pt,twoside,drfat]{article}
\usepackage{amssymb,amsthm}
\usepackage[namelimits,sumlimits,fleqn]{amsmath}
\usepackage{setspace}
\usepackage[ hmargin={2.2cm,2.4cm}, top=36mm,
    a4paper,pdftex, textheight=230mm,
    headheight=20mm, headsep=10mm, bindingoffset=0.9cm]
    {geometry}
\usepackage[small, margin=20pt]{caption}
\usepackage[usenames,dvipsnames]{color}
\definecolor{darkblue}{rgb}{0.0,0.0,0.6}
\usepackage{float}\restylefloat{figure}
\usepackage{url}
\usepackage[backref=page]{hyperref}
\hypersetup{colorlinks=true,breaklinks=true,
            linkcolor=darkblue,urlcolor=darkblue,
            anchorcolor=darkblue,citecolor=darkblue}

\usepackage{fancyhdr}
\allowdisplaybreaks[4]

\renewcommand*{\backref}[1]{}
\renewcommand*{\backrefalt}[4]{%
    \ifcase #1 (Not cited.)%
    \or        (Cited on page~#2.)%
    \else      (Cited on pages~#2.)%
    \fi}

\title{Products of Differences in Prime Order Finite Fields}
\author{Giorgis Petridis}
\date{}

\theoremstyle{plain}
\newtheorem{theorem}{Theorem}%[section]
\newtheorem{lemma}[theorem]{Lemma}
\newtheorem{proposition}[theorem]{Proposition}

\theoremstyle{definition}
\newtheorem*{definition}{Definition}
\newtheorem*{acknowledgement}{Acknowledgement}

\renewcommand*{\backref}[1]{}
\renewcommand*{\backrefalt}[4]{%
    \ifcase #1 (Not cited.)%
    \or        (Cited on page~#2.)%
    \else      (Cited on pages~#2.)%
    \fi}

\newcommand{\F}{\mathbb{F}_p} % prime finite field
\newcommand{\R}{\mathbb{R}} % finite reals
\newcommand{\supp}{\mathrm{supp}} % support
\newcommand{\ds}{\displaystyle} %displaystyle
\newcommand{\f}{\frac} % fraction
\newcommand{\sm}{\setminus} %setminus
\newcommand{\pda}{(A-A)(A-A)} % (A-A)(A-A)
\newcommand{\psd}{(A-B)(C-D)} % (A-B)(C-D)

\begin{document}

\onehalfspacing

\pagenumbering{arabic}

\setcounter{section}{0}

\bibliographystyle{plain}

\maketitle

\begin{abstract}
There exists an absolute constant $C$ with the following property. Let $A \subseteq \mathbb{F}_p$ be a set in the prime order finite field with $p$ elements. Suppose that $|A| > C p^{5/8}$. The set 
\[
(A \pm A)(A \pm A) = \{(a_1 \pm a_2)(a_3 \pm a_4) : a_1,a_2,a_3,a_4 \in A\}
\] 
contains at least $p/2$ elements. 
\end{abstract}

\section[Introduction]{Introduction}
\label{Introduction}

\let\thefootnote\relax\footnotetext{The author is supported by the NSF DMS Grant 1500984.}

Landau's $\Omega$ and $O$ notation is used above and throughout the paper. We investigate the following sum-product question in $\F$, the finite field with $p$ elements where $p$ is a prime. Let $A \subseteq \F$ be a set. When is the set
\[
(A-A)(A-A) = \{(a_1-a_2)(a_3-a_4) : a_1,a_2,a_3,a_4 \in A\} \subseteq \F
\] 
``large''? There is a number of results in the literature that address this type of question (in fact over finite fields of not necessarily prime order).
\begin{enumerate}
\item[(1)] $\pda = \F$ when $|A| > p^{3/4}$ (Hart, Iosevich, and Solymosi \cite{HIS2007}).
\item[(2)] $|\pda| = \Omega(p)$ when $|A| = \Omega( p^{2/3})$ (Bennet, Hart, Iosevich, Pakianathan, and Rudnev \cite{BHIPR}).
\item[(3)] $\underbrace{(A-A) \dots (A-A)}_\text{$k$ terms} = \F$ when $|A| = \Omega(p^{1/2})$ for some $k = O(\log\log(p))$ independent of $A$ (Balog \cite{Balog2013}).
\end{enumerate} 

The introduction of \cite{Balog2013} discusses the problem in more detail and its connection to (additive and multiplicative) characters. We strengthen (2) in the following way. 
\begin{theorem}\label{ProdDiff}
Let $p$ be a prime and $A \subseteq \F$ be a set in the finite field with $p$ elements. Suppose that $B,C,D$ are translates of non-zero dilates of $A$, that is of the form $c + \lambda A$ for $0 \neq \lambda, c \in \F$.  Suppose that $|A| = \Omega(p^{3/7})$. The number of solutions to
\[
(a-b)(c-d) = (a'-b')(c'-d') \text{ with } a,a' \in A, \dots , d,d' \in D
\]
is $~\ds \f{|A|^8}{p} + O(p^{2/3} |A|^{16/3})$.

Hence if $|A| = \Omega(p^{5/8})$, then the number of solutions is $O(|A|^8/p)$ and consequently
\[
|(A \pm A)(A \pm A)| > \frac{p}{2}.
\] 
\end{theorem}

$(A \pm A)(A \pm A)$ is used as shorthand for three different sets like
\[
(A-A)(A+A) = \{(a_1-a_2)(a_3 + a_4) : a_1,a_2,a_3,a_4 \in A\}.
\]

A number of sum-product results in finite fields can be summarised as `$|A| = \Omega(p^{2/3})$ implies that a specific set determined by sums and products of elements of $A$ is about as large as it can be'. We list the state of the art for the best known instances of the sum-product phenomenon using standard set theoretic addition such as:  $AA+AA = \{a_1 a_2 + a_3 a_4 : a_i \in A\}$  and $(A-A)^2 + (A-A)^2 = \{(a_1-a_2)^2+(a_3-a_4)^2 : a_i \in A\}.$
\begin{enumerate}
\item[(i)] $|AA+AA| = \Omega( \min\{p, |A|^{3/2}\})$ (Rudnev \cite{Rudnev}; see also \cite{Glibichuck-Konyagin2007,Hart-Iosevich2008,GPInc}).
\item[(ii)] $|(A-A)^2+(A-A)^2| = \Omega(p)$ when $|A| > p^{2/3}$ (Chapman, Erdo{\u{g}}an, Hart, Iosevich, and Koh \cite{CEHIK2012}; see also \cite{HLRN2016,GPInc}).
\item[(iii)] $|A+AA| =   \Omega( \min\{p, |A|^{3/2}\})$ (Roche-Newton, Rudnev, and Shkredov \cite{RRS}; see also \cite{Shparlinski2008,GPInc}).
\item[(vi)] $|A(A+A)| = \Omega( \min\{p, |A|^{3/2}\})$ (Aksoy-Yazici, Murphy, Rudnev, and Shkredov \cite{AYMRS}; see also \cite{GPInc}).
\end{enumerate} 

The hypothesis on $|A|$ in Theorem~\ref{ProdDiff} is below the $p^{2/3}$ threshold. To the best of our knowledge this is the first instance in the literature for such sum-product questions. It appears that the method does not generalise to the above questions, partly because addition and multiplication are not interchangeable in distributive laws. 

The proof relies on the following explicit version of Bourgain's Theorem C from \cite{Bourgain2009}, communicated to us by Misha Rudnev. Let $E_+(A, \xi A)$ be the number of solutions to $a_1+\xi a_2 = a_3 + \xi a_4$ with the $a_i \in A$ and $E_+(A) = E_+(A, A)$. For each $\xi \neq 0$, we have the inequality $E_+(A, \xi A) \leq |A|^3$. Rudnev's theorem gives strong quantitative bounds on the statement that the $|A|^3$ order of magnitude is attained for only few $\xi$.

\begin{theorem}[Rudnev]\label{Misha C}
Let $p$ be a prime and $A \subseteq \F$ and $X \subseteq \F \sm\{0\}$. Suppose that $|A|^2 |X| = O(p^2)$. The following inequality is true.
\[
\sum_{x \in X}E_+(A, x A) = O(E_+(A)^{1/2} (|A|^{3/2} |X|^{3/4} + |A| |X|)) = O(|A|^3 |X|^{3/4} + |A|^{5/2} |X|).
\]
\end{theorem}

By virtue of Bourgain's Theorem C, where it is further assumed that $|X| \leq |A|$ and the exponent of $|X|$ in the right hand side is $1 -\delta$ for some unspecified absolute positive $\delta > 0$, one could prove a weaker version of Theorem~\ref{ProdDiff} where the hypothesis would be $|A| = \Omega(p^{2/3-c})$ for some absolute positive constant $c>0$ depending on unspecified absolute constants appearing in Bourgain's paper. Rudev's explicit version of Theorem C steams from a recent result of Aksoy-Yazici, Murphy, Rudnev, and Shkredov in \cite{AYMRS}, which is based on an adaption of a theorem of Guth and Katz from their solution to the Erd\H{o}s distinct distance problem on the plane \cite{Guth-Katz2015} (see also \cite{Rudnev, RRS}). It is analogous, albeit weaker, to a result of Murphy, Roche-Newton, and Shkredov over $\R$ \cite{MRS2015}.

The first part of Theorem~\ref{ProdDiff} can be stated in terms of multiplicative characters in $\F$. For simplify we state it for the special case when $A=B=C=D$.
\[
\f{1}{p-1} \sum_{\chi \text{ mult.}} \left| \sum_{a, a' \in A} \chi(a - a') \right|^4 = \f{|A|^8}{p-1} + O(p^{2/3} |A|^{16/3}).
\]

The structure of the paper is linear. Some organisational and notational remarks: A sketch of the proof of Theorem~\ref{ProdDiff} is given in Section~\ref{Sketch}. The condition that $p$ is a prime appears in statements only when it is necessary. We denote by $\tfrac{a}{b}$ the product $a b^{-1}$.

\begin{acknowledgement}
The author would like to thank Antal Balog and Alex Iosevich for their encouragement to work on this problem, and Brendan Murphy, Misha Rudnev, and Olly Roche-Newton for generously sharing their insight.
\end{acknowledgement}

\section{Additive energies of a set and its dilates}
\label{Add En}

We begin by defining the additive energy of a set and its dilate and listing some of its properties, which will be instrumental in the proof of Theorem~\ref{ProdDiff}.

We use standard notation on set operations. For sets $A, B \subseteq \F$ we denote by $A+B = \{a+b : a\in A, b \in B\}$
their \emph{sum set}. For $\xi \in \F$ we denote by $\xi A =\{ \xi a : a \in A\}$ the \emph{dilate} of $A$ by $\xi$.

We also use \emph{representation functions} of the following kind
\[
r_{A+BA}(x) = \text{ number of solutions to $x=a+ba'$ with $a,a'\in A$ and $b \in B$}.
\]
Note that, say, the dilate sum $A + \xi A$ is the support of the function $r_{A + \xi A}$.

\begin{definition}
Let $0 \neq \xi \in \F$. The \emph{additive energy} of $A$ and $\xi A$ is defined in the following equivalent ways
\begin{align*}
E_+(A, \xi A) 
& = \sum_{x \in \F} r_{A \pm \xi A}^2(x) \\
& = \text{number of solutions to $(a_1-a_2) = \xi (a_3-a_4)$ with the $a_i \in A$}\\
& = |A|^2 + \text{number of solutions to $\frac{a_1-a_2}{a_3-a_4} = \xi$ with the $a_i \in A$}.
\end{align*}
\end{definition}

The Cauchy-Schwarz inequality and the identity $\sum_x r_{A + \xi A}(x) = |A|^2$ (each ordered pair $(a_1,a_2) \in A \times A$ contributes 1 to the sum) imply
\begin{equation}\label{CS}
E_+(A,\xi A) \geq \frac{\left(\sum_x r_{A + \xi A}(x)\right)^2}{|\supp(r_{A + \xi A})|} \geq \f{|A|^4}{p}.
\end{equation}

The proof of Theorem~\ref{ProdDiff} is based on an exact expression for the sum of the $E_+(A, \xi A)$ over all $0\neq \xi \in \F$ due to Bourgain, Katz, and Tao \cite{BKT2004}.
\begin{lemma}[Bourgain , Katz, Tao]\label{BKT}
Let $A \subseteq \F$. The following identity holds.
\begin{equation*}
\sum_{0\neq \xi \in \F} E_+(A, \xi A) =  |A|^4 + p |A|^2 - 2 |A|^3.
\end{equation*}
Hence for every set $S \subseteq \F\sm \{0\}$ we have the following inequality
\[
\sum_{\xi \in S} \left(E_+(A, \xi A) - \f{|A|^4}{p}\right)  \leq p |A|^2.
\]
\end{lemma}
\begin{proof}
The first statement is Lemma~1 in \cite{BKT2004}. The second follows by the Cauchy-Schwarz lower bound on $E_+(A,\xi A)$ stated in \eqref{CS}:
\begin{align*}
\sum_{\xi \in S} \left(E_+(A, \xi A) - \f{|A|^4}{p}\right) 
& \overset{ \eqref{CS}}\leq \sum_{0\neq \xi \in \F} \left(E_+(A, \xi A) - \f{|A|^4}{p}\right) \\
& = \sum_{0\neq \xi \in \F} E_+(A, \xi A)  -  \f{p-1}{p} |A|^4  \\
& =  |A|^4 + p |A|^2 - 2 |A|^3 - |A|^4 + \f{|A|^4}{p} \\
& \leq p|A|^2. 
\end{align*}
\end{proof} 

The Bourgain-Katz-Tao calculation implies that $E_+(A, \xi A) = \Omega(|A|^3)$ for $O(p/|A|)$ many $\xi$. We deduce stronger quantitative bounds from Theorem~\ref{Misha C}.  

\begin{lemma}\label{Large K}
Let $p$ be a prime, $A \subseteq \F$, and $K = O(\max \{p/|A| , |A|^{1/2} \})$. The number of $\xi \in \F \sm \{0\}$ such that $E_+(A, \xi A) > |A|^3 / K$ is $O(K^4)$.
\end{lemma}
\begin{proof}
Let $X$ be the set of $\xi$ in question. 

We first prove that $|A|^2 |X| = O(p^2),$ as is required in the statement of Theorem~\ref{Misha C}. The $K = O(p/|A|)$ condition implies that $E_+(A, \xi A) = \Omega(|A|^4/p)$ and so Lemma~\ref{BKT} implies that $|X| =O(p|A|^2 / (|A|^3 /K)) = O(K p / |A|)$. Therefore $|A|^2 |X| = O(K p |A|) = O(p^2)$. 

Next we use the $K = O(|A|^{1/2})$ condition to prove that $|X| < |A|^2$. Suppose not. Then the $|A|^{5/2} |X|$ term would dominate in the conclusion of Theorem~\ref{Misha C} and we would get 
\[
\f{|X| |A|^3}{K} < \sum_{\xi \in X} E_+(A, \xi A) = O(|A|^{5/2} |X|\},
\]
which would imply $K = \Omega(|A|^{1/2})$. So we must have $|X| \leq |A|^2$ and so the $|A|^{3} |X|^{3/4}$ term dominates in the conclusion of Theorem~\ref{Misha C}. This implies
 \[
|X| \f{|A|^3}{K} \leq \sum_{\xi \in X} E_+(A, \xi A) = O(|X|^{3/4} |A|^3).
\]
It follows that $|X| = O(K^4)$.
\end{proof}

Note that the $K = O(p/|A|)$ condition is not restrictive, as the proof implies.

\section{Sketch of the proof of Theorem~\ref{ProdDiff}}
\label{Sketch}

Before giving a detailed proof of Theorem~\ref{ProdDiff}, we present a quick overview of the argument when $A=B=C=D$. The intention is to illustrate why we can go bellow the $p^{2/3}$ threshold for $\pda$. There naturally is considerable overlap between this section and the subsequent, where a detailed proof is given.

$\pda$ is the support of the function $r:=r_{\pda}$, which counts the number of representations of $x \in \F$ of the form $(a_1-a_2)(a_3-a_4)$ with $a_1,\dots,a_4 \in A$. The Cauchy-Schwarz inequality implies that 
\[
|\pda| \geq \f{|A|^8}{\sum_x r(x)^2}.
\]

The heuristic that random sets should minimise the second moment of $r$ suggests that $\sum r(x)^2 \geq |A|^8/p$. Theorem~\ref{ProdDiff} follows from the fact that $|A|^8/p$ is the correct order of magnitude for $\sum r(x)^2$ provided that $|A| = \Omega(p^{5/8})$. 

To see why this is the case, we observe that, roughly speaking,
\begin{align*}
\sum_{x \in \F} r(x)^2 
& = \sum_{a_i \in A} |\{(a_1-a_2)(a_3-a_4) = (a_5-a_6)(a_7-a_8)\}| \\
& \simeq   \sum_{a_i \in A} \left|\left\{\f{a_1-a_2}{a_7-a_8}= \f{a_5-a_6}{a_3-a_4}\right\}\right| \\
& \simeq \sum_{\xi \neq 0} E_+(A,\xi A)^2.
\end{align*}
To bound $\sum E_+(A, \xi A)^2$, we express each $E_+(A,\xi A) = |A|^4/p + E_\xi$ for some $0 \leq E_\xi \leq |A|^3$. Then 
\[
\sum_{\xi \neq 0} E_+(A,\xi A)^2 \simeq \sum_{\xi \neq 0} \left(\f{|A|^4}{p}\right)^2 + \sum_{\xi \neq 0} E_\xi^2. 
\]
The first sum on the right side is about $|A|^8/p$. Our task is to bound the second sum. Let us begin with a trivial bound.
\[
\sum_{\xi \neq 0} E_\xi ^2 \leq |A|^3 \sum_{\xi \neq 0} E_\xi \overset{\text{Lemma~\ref{BKT}}}\leq p |A|^5.
\]
Noting that
\[
p |A|^5 \leq \f{|A|^8}{p} \iff |A|^3 \geq  p^2, 
\]
we observe that $|\pda| = \Omega(p)$ when $|A| = \Omega(p^{2/3})$.

This simple argument is sharp only when a positive proportion of the sum 
\[
\sum_{\xi \neq 0}  E_\xi \overset{\text{Lemma~\ref{BKT}}}\simeq p |A|^2
\]
comes from the values of $\xi$ where $E_\xi = \Omega(|A|^3)$. 

If this were the case, then there would be $\Omega(p / |A|)$ values of $\xi$ for which $E_+(A,\xi A) \simeq E_\xi = \Omega(|A|^3)$. This would contradict Lemma~\ref{Large K}. This means that $\ds \sum  E_\xi^2$ is considerably smaller than $p|A|^5$, which in turn allows one to go below the $p^{2/3}$ threshold.

\section[An explicit Theorem C]{An explicit Theorem C}
\label{Explicit C}

Let us now prove Theorem~\ref{Misha C}. For notational brevity, we express sets of solutions to equations like $(a-b)(c-d) = (a'-b')(c'-d')$ by
\[
\{(a-b)(c-d) = (a'-b')(c'-d')\}.
\]

\begin{proof}[Proof of Theorem~\ref{Misha C}]
\begin{align*}
\sum_{x \in X}E_+(A, x A) 
&  = \sum_{a_i \in A, x\in X} |\{a_1-a_2 = x(a_3 - a_4)\}| \\
&  = \sum_{t \in \F} \sum_{a_i \in A, x\in X} |\{a_1-a_2 = t = x(a_3 - a_4)\}| \\
& = \sum_{t \in \F} r_{A-A}(t) r_{X(A-A)}(t) \\
& \overset{\text{C-S}}\leq \left( \sum_{t \in \F} r_{A-A}(t)^2 \right)^{1/2} \left( \sum_{t \in \F} r_{X(A-A)}(t)^2 \right)^{1/2}\\
& = E_+(A)^{1/2} \left( \sum_{t \in \F} \sum_{a_i \in A, x_j \in X} |\{ x_1(a_1-a_2) = t = x_2(a_3-a_4)\}| \right)^{1/2} \\
& = E_+(A)^{1/2} \left( \sum_{a_i \in A, x_j \in X} |\{ x_1(a_1-a_2) = x_2(a_3-a_4)\}| \right)^{1/2}.
\end{align*}
We apply Theorem 19 in \cite{AYMRS} to bound the bracketed term under the hypothesis $|A|^2 |X| = O(p^2)$. In the notation of \cite{AYMRS}, we apply Theorem 19 (or the discussion above it) to $B=X$ and $C=A$. So $P = A \times X$ and hence $|L|= |A| |X|$. We note that $|L|  |A| = |A|^2 |X| = O(p^2)$. Therefore the number of solutions to $x_1(a_1-a_2) = x_2(a_3-a_4)$ with the $a_i \in A$ and $x_j \in X$ is 
\[
O(|A|^3 |X|^{3/2} + |A|^2 |X|  (|A| + |X|))  = O(|A|^3 |X|^{3/2} + |A|^2 |X|^2).
\]
The theorem follows by the inequalities $(m+n)^{1/2} \leq m^{1/2} + n^{1/2}$ and $E_+(A) \leq |A|^3$.
\end{proof}

Murphy communicated to us the following analogous result for \emph{multiplicative energies}, which, like Theorem~\ref{Misha C}, is analogous, albeit weaker, to a result of Murphy, Roche-Newton, and Shkredov from \cite{MRS2015}. We define the multiplicative energy of $A$ and its translate $A+\xi$ as follows.
\[
E_\times(A, A + \xi ) = \text{ number of solutions to } a_1(a_2 + \xi) = a_3 (a_4 + \xi) \text{ with the } a_i \in A.
\] 
The multiplicative energy of $A$ is defined by $E_\times(A) = E_\times(A,A)$ and is at most $|A|^3$.
\begin{theorem}[Murphy]\label{Brendan C}
Let $p$ be a prime and $A, X \subseteq \F$. Suppose that $|A|^2 |X| = O(p^2)$. The following inequality is true.
\[
\sum_{x \in X}E_\times(A, A+x) = O(E_\times(A)^{1/2}(|A|^{3/2} |X|^{3/4} + |X| |A|)) = O(|X|^{3/4} |A|^3 + |X| |A|^{5/2}).
\]
\end{theorem}

\section[Bringing additive energies into the calculation]{Bringing additive energies into the calculation}
\label{Intermediate}

To prove Theorem~\ref{ProdDiff} we express $(A-B)(C-D)$ as the support of the representation function $r:=r_{(A-B)(C-D)}$ and bound the second moment of $r$, which equals the number of solutions to the given equation. The rest follows from the Cauchy-Schwarz inequality.

The proof is split in two parts: The first where we reduce the question to bounding $\ds \sum_{\xi \neq 0} E_+(A, \xi A)^2$; and the second where via Theorem~\ref{Misha C} and Lemma~\ref{BKT} we bound the sum $\ds \ds \sum_{\xi \neq 0} E_+(A, \xi A)^2$. In this section we complete the first step. 

Recall that $B, C,D$ are translates of non-zero dilates of $A$ and that for $x\in \F$ we take 
\begin{equation}\label{r}
r(x) = \text{ number of solutions to $(a-b)(c-d) =x$ with $a \in A, \dots, d \in D$}.
\end{equation}  

\begin{lemma}\label{AddEn}
Let $A \subseteq \F$ and $B, C,D$ be translates of non-zero dilates of $A$. The number of solutions to
\[
(a-b)(c-d) = (a'-b')(c'-d') \text{ with } a,a' \in A, \dots, d,d' \in D
\]
is $\ds O(|A|^6) +  \sum_{\xi \neq 0} E_+(A , \xi A)^2$.
\end{lemma}

\begin{proof}
We denote the set of solutions to $(a-b)(c-d) = (a'-b')(c'-d')$ by
\[
\{(a-b)(c-d) = (a'-b')(c'-d')\}.
\]

The number of solutions where $(a-b)(c-d) = (a'-b')(c'-d') = 0$ is $O(|A|^6)$. From now on we consider non-zero solutions. To denote this we put a $*$ over sums. We have the following string of inequalities.
\begin{align}\label{1}
& \sum_{a,a' \in A} \sum_{b,b' \in B} \sum_{c,c' \in C} \sum_{d,d' \in D} | \{ (a-b)(c-d) = (a'-b')(c'-d') \neq 0\} | \nonumber \\
& ~= \sum_{a,a' \in A} \sum_{b,b' \in B}^* \sum_{c,c' \in C} \sum_{d,d' \in D}^* | \{ (a-b)(c-d) = (a'-b')(c'-d') \} | \nonumber \\
& ~= \sum_{a,a' \in A} \sum_{b,b' \in B}^* \sum_{c,c' \in C} \sum_{d,d' \in D}^* \left|\left\{\f{a-b}{a'-b'} = \f{c-d}{c'-d'}\right\}\right| \nonumber \\
& ~= \sum_{\xi \neq 0} \sum_{a,a' \in A} \sum_{b,b' \in B}^* \sum_{c,c' \in C} \sum_{d,d' \in D}^* \left| \left \{\f{a-b}{a'-b'} = \xi = \f{c-d}{c'-d'} \right\} \right| \nonumber \\
& ~= \sum_{\xi \neq 0} \left(\sum_{a,a' \in A} \sum_{b,b' \in B}^*\left| \left\{ \f{a-b}{a'-b'} = \xi \right\} \right|\right) \left(\sum_{c,c' \in C} \sum_{d,d' \in D}^* \left| \left\{ \f{c-d}{c'-d'} = \xi \right\} \right| \right).
\end{align}

For $\xi \neq 0$,
\begin{align*}
\sum_{a,a' \in A} \sum_{b,b' \in B}^*\left| \left\{ \f{a-b}{a'-b'} = \xi \right\} \right|
& = \sum_{a,a' \in A} \sum_{b,b' \in B}^* |\{a-\xi a' = b - \xi b'\}| \\
& = \sum_x \sum_{a,a' \in A} \sum_{b,b' \in B}^* |\{a-\xi a' = x =b - \xi b'\}| \\
& \leq \sum_x r_{A - \xi A}(x) r_{B - \xi B}(x)\\
& \leq \sqrt{ \left(\sum_x r_{A - \xi A}(x)^2\right) \left( \sum_xr_{B - \xi B}(x)^2\right) } \\
& = \sqrt{E_+(A, \xi A) E_+(B, \xi B)} \\
& = E_+(A, \xi A).
 \end{align*}
In the last step we used the fact that if $B = c + \lambda A$, then $E_+(A) = E_+(B)$. 

We similarly have 
\[
\sum_{c,c' \in C} \sum_{d,d' \in D}^* \left| \left\{ \f{c-d}{c'-d'} = \xi \right\} \right|   \leq E_+(A, \xi A).
\]
So inequality \eqref{1} becomes 
\[
\sum_{a,a' \in A} \sum_{b,b' \in B} \sum_{c,c' \in C} \sum_{d,d' \in D} | \{ (a-b)(c-d) = (a'-b')(c'-d') \neq 0\} \leq \sum_{\xi \neq 0} E_+(A, \xi A)^2.
\]
The lemma follows by adding the ``$\xi = 0$ term''.
\end{proof}

We continue with some more rearranging.

\begin{lemma}\label{Exi2}
Let $A \subseteq \F$. The following identity holds.
\[
\sum_{\xi \neq 0} E_+(A,\xi A)^2 \leq \f{|A|^8}{p} + 2 |A|^6 + \sum_{\xi \neq 0} \left(E_+(A,\xi A) - \f{|A|^4}{p}\right)^2 .
\]
\end{lemma}
\begin{proof}
\begin{align*}
\sum_{\xi \neq 0} \left(E_+(A,\xi A) - \f{|A|^4}{p}\right)^2 
& = \sum_{\xi \neq 0} E_+(A,\xi A)^2 - 2 \f{|A|^4}{p} \sum_{\xi \neq 0} E_+(A , \xi A) + (p-1) \f{|A|^8}{p^2} \\
& \overset{\text{Lemma~\ref{BKT}}}= \sum_{\xi \neq 0} E_+(A,\xi A)^2 - \f{|A|^8}{p}  - 2 |A|^6 + 4 \f{|A|^7}{p} - \f{|A|^8}{p^2}\\
& \geq \sum_{\xi \neq 0} E_+(A,\xi A)^2 - 2 |A|^6 - \f{|A|^8}{p}. \hskip 10ex  \qedhere
\end{align*} 
\end{proof}

\section{The sum of the squares of additive energies}
\label{SumSq}

The next and more substantial step is to use Lemma~\ref{Large K} to, essentially, bound the sum of the squares of the additive energies.

\begin{proposition}
Let $p$ be a prime and $A \subseteq \F$. The following inequality holds.
\[
\sum_{\xi \neq 0} \left(E_+(A,\xi A) - \f{|A|^4}{p}\right)^2  =  O(p^{2/3} |A|^{16/3} + p |A|^{9/2}).
\]
\end{proposition}

\begin{proof}
We split the calculation in three parts. Let $K = O(\max\{p/|A|,|A|^{1/2})$ be a parameter to be determined later.

For ``very small'' $E_+(A,\xi A)$ we use Lemma~\ref{BKT}.
\begin{equation}\label{vsmall}
\sum_{E_+(A,\xi A) \leq  |A|^{5/2}} \left(E_+(A,\xi A) - \f{|A|^4}{p}\right)^2 \leq |A|^{5/2}\sum_{\xi \neq 0} \left(E_+(A,\xi A) - \f{|A|^4}{p}\right) \overset{\text{Lem.\,\ref{BKT}}}= p |A|^{9/2}.
\end{equation}

Similarly, for ``small'' $E_+(A,\xi A)$ we have
\begin{equation}\label{small}
\sum_{E_+(A,\xi A) \leq\f{ |A|^3}{K}} \left(E_+(A,\xi A) - \f{|A|^4}{p}\right)^2 \leq \f{p |A|^5}{K}.
\end{equation}

For ``large'' $E_+(A, A\xi)$ we define an integer $k$ via  $2^{k-1} < K \leq 2^k$ and consider sets $X_1, \dots, X_k$ defined by
\[
X_i = \left\{ \xi \neq 0 : \f{|A|^3}{2^i} < E_+(A, \xi A) \leq \f{|A|^3}{2^{i-1}}\right\}.
\]
By Lemma~\ref{Large K} we conclude $|X_i| = O(2^{4i})$. So
\begin{align}\label{large}
\sum_{E_+(A, \xi A) > \tfrac{|A|^3}{K}}  \left(E_+(A, \xi A) - \f{|A|^4}{p} \right)^2  \nonumber
& \leq \sum_{E_+(A, \xi A) > \tfrac{|A|^3}{K}}  E_+(A, \xi A)^2 \nonumber \\
& = \sum_{i=1}^k \sum_{\xi \in X_i} E_+(A, \xi A)^2 \nonumber \\
& \leq \sum_{i=1}^k |X_i| \f{|A|^6}{2^{2i}} \nonumber \\
& = O\left( |A|^6 \sum_{i=1}^k 2^{2i} \right)  \nonumber \\
&= O(2^{2k} |A|^6)  \nonumber \\
&= O(K^2 |A|^6).
\end{align}

Therefore
\begin{align*}
\sum_{\xi \neq 0} \left(E_+(A,\xi A) - \f{|A|^4}{p}\right)^2 
& \leq \sum_{E_+(A, \xi A) \leq  |A|^{5/2}}  \left(E_+(A, \xi A) - \f{|A|^4}{p} \right)^2 \\
& \hskip 5ex + \sum_{E_+(A, \xi A) \leq  \tfrac{|A|^3}{K}}  \left(E_+(A, \xi A) - \f{|A|^4}{p} \right)^2 \\
& \hskip 5ex + \sum_{E_+(A, \xi A) > \tfrac{|A|^3}{K}}  \left(E_+(A, \xi A) - \f{|A|^4}{p} \right)^2 \\
& \overset{\eqref{vsmall},\eqref{small}, \eqref{large}}  = O\left( p |A|^{9/2} +  \f{p |A|^5}{K} + |A|^6 K^2 \right).
\end{align*}

To optimise $K$ we equate the two expressions containing it.
\[
\f{p |A|^5}{K} = |A|^6 K^2 \iff K^3 = \f{p}{|A|} \iff K = \left(\f{p}{|A|}\right)^{1/3},
\]
crucially noting that the above value is indeed $O(\sqrt{p/|A|})$. We have already assumed that $K = O(|A|^{1/2})$.

Taking $K = (p/|A|)^{1/3}$ produces the stated upper bound.
\end{proof}

\section{Conclusion of the proof of Theorem~\ref{ProdDiff}}
\label{Proof}

We are finally in position to prove Theorem~\ref{ProdDiff}.

\begin{proof}[Proof of Theorem~\ref{ProdDiff}]
By Lemmata~\ref{AddEn} and \ref{Exi2}, the number of solutions to
\[
(a-b)(c-d) = (a'-b')(c'-d') \text{ with } a,a' \in A, \dots , d,d' \in D
\]
is $\ds \f{|A|^8}{p} + O(p^{2/3} |A|^{16/3} + p |A|^{9/2} + |A|^6) =  \f{|A|^8}{p} + O(p^{2/3} |A|^{16/3} + p |A|^{9/2})$.

When $|A| = \Omega(p^{5/8})$, the first term dominates the second and third.

For the final conclusion we note that $\psd$ is the support of the function $r$ defined in \eqref{r} on p.~\pageref{r}. By the Cauchy-Schwarz inequality 
\[
|\psd| \geq \f{\left( \sum_{x \in \F} r(x) \right)^2}{\sum_{x \in \F} r(x)^2} =  \f{|A|^8}{\sum_{x \in \F} r(x)^2} 
\]
because each ordered quadruple $(a,b,c,d) \in A \times B \times C \times D$ contributes 1 to the sum $\sum_x r(x)$. 

In the $|A|=\Omega(p^{5/8})$ range, the sum $\sum_x r(x)^2 = O(|A|^8/p)$ because it is nothing other than the number of solutions to
\[
(a-b)(c-d) = (a'-b')(c'-d') \text{ with } a,a' \in A, \dots , d,d' \in D.
\]
\end{proof}

\phantomsection

\addcontentsline{toc}{section}{References}

\bibliography{all}

\hspace{20pt} Department of Mathematics, University of Rochester, New York, USA.

\hspace{20pt} \textit{Email address}: \href{mailto:giorgis@cantab.net}{giorgis@cantab.net}

\end{document}